\documentclass{amsart}
\usepackage[utf8]{inputenc}

\usepackage{mathtools}
\usepackage{amssymb}
\usepackage{comment}
\usepackage{xcolor} 
\usepackage{hyperref}
\hypersetup{
   colorlinks,
   linkcolor={red!70!black},
   citecolor={blue!70!black},
   urlcolor={blue!90!black}
}

\usepackage{esint}
\usepackage{fourier}

\mathtoolsset{mathic}

\newcommand{\cC}{\mathcal{C}}

\newcommand{\cH}{\mathcal{H}}

\newcommand{\cN}{\mathcal{N}}

\renewcommand{\Vert}{\mathsf{VF}}
\newcommand{\Aff}{\mathsf{Aff}}
\newcommand{\SAff}{\mathsf{SAff}}

\DeclareMathOperator{\Cone}{Cone}

\DeclareMathOperator{\Lip}{Lip}
\DeclareMathOperator{\supp}{supp}
\DeclareMathOperator{\Span}{span}

\renewcommand{\H}{\mathbb{H}}
\newcommand{\R}{\mathbb{R}}
\newcommand{\Z}{\mathbb{Z}}
\newcommand{\lip}{\mathrm{Lip}}
\newcommand{\from}{\colon}

\newcommand{\zero}{\mathbf{0}}
\newcommand{\ud}[0]{\,\mathrm{d}}

\newtheorem{thm}{Theorem}[section]

\newtheorem{lemma}[thm]{Lemma}
\newtheorem{cor}[thm]{Corollary}
\newtheorem{prop}[thm]{Proposition}
\theoremstyle{remark}

\newtheorem*{roadmap}{Roadmap}

\DeclareMathOperator{\diam}{diam}

\title{The strong geometric lemma for intrinsic Lipschitz graphs in Heisenberg groups}
\author{Vasileios Chousionis}
\address{Department of Mathematics, University of Connecticut}
\email{vasileios.chousionis@uconn.edu}

\author{Sean Li}
\address{Department of Mathematics, University of Connecticut}
\email{sean.li@uconn.edu}

\author{Robert Young}
\address{Courant Institute of Mathematical Sciences, New York University}
\email{ryoung@cims.nyu.edu}

\thanks{V.~C.\ was supported by Simons Foundation Collaboration grant 521845. R.~Y.\ was supported by NSF grant 1612061}
\begin{document}
\maketitle
\begin{abstract}
    We show that the $\beta$--numbers of intrinsic Lipschitz graphs of Heisenberg groups $\H_n$ are locally Carleson integrable when $n \geq 2$.  Our technique relies on a recent Dorronsoro inequality \cite{FO} as well as a novel slicing argument.   A key ingredient in our proof is a Euclidean inequality bounding the $\beta$--number of a function on a cube of $\R^n$ using the $\beta$--number of the restriction of the function to codimension--1 slices of the cube. 
\end{abstract}
\section{Introduction}
In \cite{JonesTSP}, Jones characterized subsets $E \subset \R^2$ that lie on finite length rectifiable curves in the now-famous traveling salesman theorem.  Given a ball $B(x,r) \subset \R^2$, he introduced the quantity
\begin{align*}
  \beta_E(x,r) = \inf_L \sup_{z \in B(x,r) \cap E} \frac{d(z,L)}{r},
\end{align*}
where the infimum is over all affine lines $L$ in $\R^2$.  This quantity, known as the $\beta$--number, is a scale-invariant measure of how close $E \cap B(x,r)$ is to a line.  Jones proved the following theorem.
\begin{thm}
\label{tstj}
  A subset $E \subset \R^2$ lies on a finite length rectifiable curve if and only if
  \begin{align}
   \diam(E)+ \int_0^\infty \int_{\R^2} \beta_E(x,r)^2 \ud x \;\frac{\ud r}{r^2} < \infty. \label{e:carleson}
  \end{align}
\end{thm}
Note that the measure $\frac{\ud r}{r^2}$ has a singularity at $r = 0$, so \eqref{e:carleson} can be viewed as saying $E$ is close to a line ($\beta_E(x,r)$ is small) for ``most'' balls according to the measure $\ud x \frac{\ud r}{r^2}$.  Thus, this can be viewed as a quantitative version of the classical Rademacher's theorem, which implies that rectifiable curves are close to affine on infinitesimal scales.  Integrals against measures of the form $\ud x\;\frac{\ud r}{r^n}$ are known as Carleson integrals.  

Jones, in \cite{JonesSquare}, used the quantitative geometric information provided by \eqref{e:carleson} to give another proof of the $L_2$--boundedness of the Cauchy transform on $1$\nobreakdash--dimensional Lipschitz graphs. Since Jones's work, $\beta$--numbers and their variants have appeared frequently in harmonic analysis, geometric measure theory, and related fields. In particular, Theorem~\ref{tstj} has been generalized to various other spaces and settings \cite{Oki, SchulTSP, LiSchul, LiSchul2, FFP, CLZ,Li-strat, DavidSchul}. Moreover, the introduction of Carleson integrals of $\beta$--numbers was a starting point for the theory of quantitative rectifiability of David and Semmes \cite{DS1,DS2}. Briefly, this theory provides several characterizations of Ahlfors regular sets of $\R^n$ satisfying a Carleson bound similar to \eqref{e:carleson} via a suite of geometric and analytic conditions. Notably, it serves as a broad geometric framework for singular integrals acting on lower dimensional subsets of $\R^n$.




One of the settings where generalizations of the traveling salesman theorem have been studied is the setting of Carnot groups \cite{LiSchul, LiSchul2, FFP, CLZ,Li-strat}.  The Carnot groups are a class of nilpotent Lie groups whose abelian members are precisely Euclidean spaces.  Thus, they can be viewed as generalizations of $\R^n$.  The simplest examples of nonabelian Carnot groups are the Heisenberg groups $\H_n$. Extending Theorem~\ref{tstj} and other aspects of quantitative rectifiability in Carnot groups contributes to the development of geometric measure theory on these sub-Riemannian spaces.  For a recent overview of this research program, which started about 20 years ago, we refer the reader to the lecture notes \cite{SCnotes}.

An interesting challenge for researchers in sub-Riemannian geometric measure theory is to develop a robust theory of sub-Riemannian rectifiability. For example, the attempt to define rectifiability in the $(2n+1)$--dimensional Heisenberg group $\H_n$ modeled after Federer's classical definition using Lipschitz images is problematic; Ambrosio and Kirchheim \cite{AK} proved that a Lipschitz image $f(\R^k )\subset \H_n$ has zero $k$--dimensional Hausdorff measure for $n+1 \leq k \leq 2n$.  This indicates that the notion of rectifiability changes drastically between the low-dimensional and the low-codimensional case. 

Franchi, Serapioni and Serra-Cassano  \cite{FSS1} introduced intrinsic Lipschitz graphs in order to define rectifiable sets of codimension 1 in $\H_n$. These are sets satisfying a cone condition, similar to the one satisfied by Euclidean Lipschitz graphs.   The cone condition will be defined in the next section, but a notable class of examples is given by the fact that level sets of $C^1$ functions are locally intrinsic Lipschitz graphs away from critical points. Furthermore, just as Rademacher's theorem implies that Lipschitz graphs in $\R^n$ can be approximated by planes almost everywhere, intrinsic Lipschitz graphs can be locally approximated by vertical planes almost everywhere \cite{FSSCDiff}. The approach initiated in \cite{FSS1} has been quite successful, as several important subsequent contributions \cite{MSS,FSSCDiff} laid down the foundations for a meaningful theory of low-codimensional rectifiability. 

Recently, intrinsic Lipschitz graphs have also been used to study quantitative rectifiability. This approach was introduced independently in \cite{CFO1,NY1}, which give quantitative bounds on how close intrinsic Lipschitz graphs are to vertical planes and vertical sets. See also \cite{NY2,FORig,Rig} for some related recent results. One motivation for this work is that, as in the Euclidean case, the quantitative rectifiability of intrinsic Lipschitz graphs is related to the behavior of a natural singular integral arising in the study of removability for Lipschitz harmonic functions in $\H_n$, see \cite{CFO2,FORiesz}. 

In this paper, we bound the  quantitative rectifiability of intrinsic Lipschitz graphs by proving bounds on the $\beta$--numbers of intrinsic Lipschitz graphs analogous to the bounds in Theorem~\ref{tstj}.
Given $E \subset \H_n$ (typically an intrinsic Lipschitz graph) and \(x\in \H_n\), we define the codimension--1 version of the $\beta$--numbers as
\begin{equation}\label{eq:def beta} \beta_E(x,r)=\inf_{L \in \mathsf{VP}} \left[ r^{-2n-1} \int_{B(x,r)\cap E} \left(\frac{d(y,L)}{r}\right)^2  \; \ud  \mathcal{H}^{2n+1}(y)\right]^{1/2}
\end{equation}
where in the infimum, $L$ ranges over the set of vertical planes (planes containing the center) in $\H_n$.  This is scale-invariant in the sense that $\beta_{\delta_t(E)}(\delta_t(x),tr)=\beta_{E}(x,r)$ for any $t>0$. 

Codimension--1 $\beta$--numbers in the Heisenberg group were previously considered in \cite{CFO1}, where it was shown that a weaker qualitative analogue of \eqref{e:carleson}, known as the
weak geometric lemma, holds for intrinsic Lipschitz graphs in $\H_n$. This result was written only for $\H_1$ in \cite{CFO1}, but the proof extends to all $\H_n$. In this paper we will show that intrinsic Lipschitz graphs in $\H_n$, for $n \geq 2$, satisfy a local version of \eqref{e:carleson} known as the (strong) geometric lemma. We note that the (strong) geometric lemma implies the weak geometric lemma, see \cite[Part I, 1.1.5]{DS2} for a relevant discussion in Euclidean spaces. 

Our main result is the following theorem.
\begin{thm}\label{thm:beta2 highdim}
  Let $n\ge 2$ and let $\Gamma$ be an intrinsic $\lambda$--Lipschitz graph in $\H_n$.  Then, for any $R>0$ and any ball $B=B(y,R)\subset \H_n$,
  \begin{align}
  \int_0^R \int_{B \cap \Gamma} \beta_{\Gamma}(x,r)^2 \ud \mathcal{H}^{2n+1}(x) \frac{\ud r}{r} \lesssim_{\lambda} R^{2n+1}. \label{e:main}
  \end{align}
\end{thm}
This generalizes Dorronsoro's Theorem, which proves a similar inequality for Lipschitz graphs in $\R^n$.  More generally, an Ahlfors regular subset of $\R^n$ satisfies a Euclidean analogue of \eqref{e:main} if and only if it is uniformly rectifiable.

We prove Theorem \ref{thm:beta2 highdim} using a slicing technique. Though the intrinsic Lipschitz condition is nonlinear, a codimension--1 slice of an intrinsic Lipschitz graph can be viewed as the graph of a Lipschitz function on a subset of $V_0$ that is isomorphic to $\H_{n-1}$.
Fässler and Orponen \cite{FO} proved a version of Dorronsoro's Theorem for Lipschitz functions on Heisenberg groups which gives bounds analogous to \eqref{e:main}.  
This shows that a version of Theorem~\ref{thm:beta2 highdim} holds on each slice, and we show that if Theorem~\ref{thm:beta2 highdim} holds on a rich enough family of slices of $\Gamma$, then it holds on all of $\Gamma$. A similar slicing argument in the Euclidean setting was used recently in \cite{OR}.

Our slicing technique relies on passing from $\H_n$ to a lower-dimensional Heisenberg group $\H_{n-1}$, which is impossible when $n=1$.  This leaves open the problem of whether Theorem~\ref{thm:beta2 highdim} is true for $\H_1$, which we plan to investigate in future work.  

\begin{roadmap}
In Section~\ref{sec:prelims}, we introduce some basic notation and definition for intrinsic Lipschitz graphs.  To prove Theorem~\ref{thm:beta2 highdim}, we need some results on balls in intrinsic graphs and their projections and a description of codimension--1 slices of intrinsic graphs, which we present in Section~\ref{sec:slices}.  We will also need to pass between the non-parametric $\beta$--numbers defined in \eqref{eq:def beta} (which are defined in terms of a graph $\Gamma$) and parametric $\beta$--numbers (defined in terms of a function $f$).  We prove the necessary lemmas in Section~\ref{sec:comparing}.   
Finally, in Section~\ref{sec:proof of main}, we prove Theorem~\ref{thm:beta2 highdim}, modulo one last lemma that bounds the $\beta$--numbers of a graph in terms of the $\beta$--numbers of its slices, and in Section~\ref{sec:slice affine}, we prove this last lemma.
\end{roadmap}

\section{Preliminaries}\label{sec:prelims}

We define the Heisenberg group $\H_n$ as the Lie group $(\R^{2n+1},\cdot)$ where elements of $\R^{2n+1}$ are written as $(x,y,z)$ for $x,y \in \R^n$ and $z \in \R$.  The group product is defined as
\begin{align*}
  (x,y,z) \cdot (x',y',z') = (x+x',y+y',z+z' + \Omega((x,y),(x',y'))/2)
\end{align*}
where $\Omega((x,y),(x',y')) = \sum_{i=1}^n (x_iy_i' - x_i'y_i)$ is the symplectic form on $\R^{2n}$.  Note that two elements $(x,y,z), (x',y',z') \in \H_n$ do not commute unless $\Omega((x,y),(x',y')) = 0$. We will also use the standard commutator notation $[u,v]=uvu^{-1}v^{-1}$ for $u,v \in \H_n$. The identity element in $\H_n$ is $\zero=(0,0,0)$. 

Let $X_1,\dots, X_n, Y_1,\dots, Y_n, Z$ be the coordinate vectors of $\H_n$ and let $x_1,\dots, x_n$, $y_1,\dots, y_n$, $z\from \H_n\to \R$ be the coordinate functions.
The center of the group is $\langle Z \rangle = \{ (0,0,z) : z \in \R\}$.  A subgroup $W\subset \H_n$ is \emph{vertical} if it contains the center $\langle Z\rangle$.
An element $w \in \H_n$ is said to be a {\it horizontal vector} if $z(w) = 0$, and we let $A$ be the set of horizontal vectors.  Let $d$ be the Carnot--Carathéodory metric on $\H_n$. For any $h \in \H_n$ we let $\|h\|=d(\zero,h)$.

We define a family of automorphisms
\begin{align*}
  \delta_t \from  \H_n &\to \H_n \\
  (x,y,z) &\mapsto (t x, t y, t^2 z)
\end{align*}
for $t\in \R$.
The $\delta_t$ dilate the metric in that for any $g,h \in \H_n$, we have
\begin{align*}
  d(\delta_t(g),\delta_t(h)) = |t| d(g,h).
\end{align*}
For any horizontal vector $w \in \H_n$ and $\alpha \in \R$, we define $w^\alpha = \delta_\alpha(w)$; when $\alpha\in \Z$, this agrees with the usual notion of exponentiation.

The projection
\begin{align*}
  \pi \from  \H_n &\to \R^{2n} \\
  (x,y,z) &\mapsto (x,y)
\end{align*}
is a Lipschitz homomorphism.   A function $f \from  \H_n \to \R^k$ is \emph{affine} if there is an affine function $\tau \from  \R^{2n} \to \R^k$ so that $f = \tau \circ \pi$.

A \emph{vertical plane} $V$ is a subset $V=\pi^{-1}(P)\subset \H_n$ where $P \subset \R^{2n}$ is a $(2n-1)$--dimensional affine plane.  One particular vertical plane, which we'll use very often is $V_0 = \{y_n = 0\}$. A function $f\from V_0\to \R$ is \emph{vertical} if it is constant on cosets of $\langle Z\rangle$. We will denote by $\Vert$ the set of vertical functions on $V_0$, and by $\Aff$ the set of vertical affine functions on $V_0$; i.e., restrictions to $V_0$ of affine function $\H_n \to \R$.  Note that these functions are of the form
$$f(v)=\sum_{i=1}^n \alpha_i x_i(v) + \sum_{i=1}^{n-1} \beta_i y_i(v) + \gamma.$$

Let $w \in A$ be a horizontal vector such that $y_n(w)=1$, we let $\Pi_w\from \H_n\to V_0$ be the map such that $\Pi_w(h)$ is the unique point of intersection of the coset \(h\langle w\rangle\) and the vertical plane \(V_0\).  This map projects  $\H_n$ to $V_0$ along cosets of $\langle w\rangle$, and we have $\Pi_w(h)=hw^{-y_n(h)}$.
For any $g, h\in \H_n$,
\begin{equation}\label{eq:Pi invariance}
\Pi_w(gh)=gh\cdot w^{-y_n(g)-y_n(h)}=g\left(h \cdot w^{-y_n(h)}\right)\cdot w^{-y_n(g)} = \Pi_w(g\Pi_w(h)).
\end{equation}

For any function $f\from V_0\to \R$ and any horizontal vector $w$ with $y_n(w)=1$, we define the $w$--intrinsic graph of $f$ as
$$\Gamma_{f,w}=\{v w^{f(v)} : v\in V_0\}.$$
For $0<\lambda<1$, let $\Cone_\lambda$ be the open double cone
\begin{align*}
  \Cone_\lambda = \{p\in \H_n : \lambda d(\mathbf{0},p) < |y_n(p)|\}.
\end{align*}
We say that a subset $\Gamma\subset \H_n$ is an \emph{intrinsic $\lambda$--Lipschitz graph} if for every $x\in \Gamma$, $(x\cdot \Cone_\lambda) \cap \Gamma=\emptyset$.  For any such $\Gamma$ and any $w\in \Cone_\lambda$ with $y_n(w)=1$, the restriction $\Pi_w|_{\Gamma}$ is injective, so we can define a function $f_w\from \Pi_w(\Gamma)\to \R$, $f_w(\Pi_w(p))=y_n(p)$ such that $\Gamma=\Gamma_{f_w,w}$.  Conversely, if $\Gamma_{f, w}$ is an intrinsic $\lambda$--Lipschitz graph, we say that $f$ is an \emph{intrinsic $(w,\lambda)$--Lipschitz function}.

We finally record that that the Hausdorff (2n+1)-measure $\cH^{2n+1}$ is an Ahlfors (2n+1)-regular measure when restricted to an intrinsic $\lambda$--Lipschitz graph $\Gamma$, see e.g. \cite[Theorem 3.9]{FS}.  That is, there exists a constant $C > 1$ depending only on $\lambda$ so that
\begin{align*}
  C^{-1}r^{2n+1} \leq \cH^{2n+1}(B(x,r) \cap \Gamma) \leq Cr^{2n+1}, \qquad \forall x \in \Gamma, r > 0.
\end{align*}

\section{Projections and slices of intrinsic Lipschitz graphs}\label{sec:slices}

We will sometimes need to pass between balls in $\H_n$ and their projections to $V_0$, which can be highly distorted.  In this section, we define some quasiballs that will make this more convenient.  Moreover, we will introduce a slicing method in order to obtain a family of decompositions of an intrinsic Lipschitz graph $\Gamma$ into graphs of real valued Lipschitz functions with domain $\H_{n-1}$.

We start with some notation. For any subspace $S\subset A$, let $S^\Omega$ be the symplectic complement of $S$, i.e., 
$$S^\Omega=\{v\in A : \Omega(v,s)=0 \text{ for all }s\in S\}.$$
The subspace $S\subset A$ is called symplectic if $S \cap S^\Omega= \emptyset$. In a similar manner, for any subspace $S \subset A$ we define
$$S^{\overline{\Omega}}= \{h\in \H_n: \overline{\Omega}(h,s)=0 \mbox{ for all }s \in S\},$$
where $\overline{\Omega}\from \H_n\times \H_n\to \R$ is the the alternating form $\overline{\Omega}(g,h)=\Omega(\pi(g),\pi(h))$.  Note that $[g,h]=\overline{\Omega}(g,h)Z$ for all $g,h \in \H_n$.



 Slightly abusing notation, for a vector $w\in A$ we let $$v^\Omega=\langle w\rangle^\Omega  \mbox{ and }v^{\overline{\Omega}}=\langle w\rangle^{\overline{\Omega}} =\pi^{-1}(w^\Omega).$$ Note that since $\Omega(w,w)=0$, we have $w\in w^\Omega$. It is well known that for any subspace $S \subset A$ it holds that $\dim S+\dim S^\Omega=2n$, see e.g. \cite[Chapter 1]{berndt}.
Hence, for every horizontal vector $w$ such that $y_n(w)=1$, the complement $w^{\Omega}$ is a $(2n-1)$--dimensional horizontal subspace that contains $w$. For such $w$, we let $C_w=V_0\cap w^{\Omega}$; this is a $(2n-2)$--dimensional horizontal subspace of $V_0$.  Note that $C_w$ does not uniquely determine $w$, since we have $C_w=C_{w+tX_n}$ for any $t\in \R$.  Let also 
$$P_w=V_0\cap w^{\overline{\Omega}}=C_w+\langle Z\rangle,$$ and observe that this is a $(2n-1)$--dimensional vertical subspace of $V_0$.

Let $w$ and $C_w$ be as above and let $\nu\in A\cap V_0$ be a horizontal unit vector orthogonal to $C_w$; this is unique up to sign.  Let
$$R_w = \{ s \nu + p + t Z \in V_0 : p \in B_{C_w}(\mathbf{0},1),|s|\le 1, |t|\le 1\},$$
where $B_{C_w}(\mathbf{0},1)$ is the unit ball in $C_w$.  For any $g \in \H_n$ and any $r > 0$, we define
$$Q_w(g,r)=\Pi_w(g\delta_r(R_w)).$$
Since $R_w\subset V_0$, when $g=\zero$, we simply have $Q_w(\zero,r)=\delta_r(R_w)$.  Note that the map $x\mapsto \Pi_w(gx)$ preserves Lebesgue measure on $V_0$, so if $\mu$ is Lebesgue measure, then
\begin{equation}\label{eq:measure of Q}
  \mu(Q_w(g,r))=\mu(\delta_r(R_w))\approx r^{2n+1}.
\end{equation}

We now recall the definition of quasiballs. Let $(X, \rho)$ be a metric space and let $\lambda \geq 1$.  Recall that a $\lambda$--quasiball (or simply a quasiball if $\lambda$ is understood) is any set $E \subset X$ for which there exist $x \in X$ and  $R > 0$ for which
\begin{align*}
  B_{\rho}(x,R) \subseteq E \subseteq B_{\rho}(x,\lambda R).
\end{align*}
In the next lemma we'll show that when we slice $Q_w(g,r)$ with cosets of $P_w$ we obtain quasiballs. 
\begin{lemma}\label{lem:slices are quasi}
 Let $w \in A$ with $y_n(w)=1$. For any $g \in \H_n$, any $r > 0$, and any $u\in V_0$, the intersection $Q_w(g,r) \cap u P_w$ is either a quasiball of radius $\approx r$ or empty.
\end{lemma}

\begin{proof}
  Let $R_w$ and $\nu$ be as above.
  We first consider the case $g=\zero$. 
  For $a\in \R$, let $D_{a}=R_w \cap \nu^a P_w$.
  When $|a|\le 1$, 
  $$D_{a}=\{a\nu + p + t Z : p \in B_{C_w}(\mathbf{0},1), |t|\le 1\}$$
  is a quasiball of radius $\approx 1$ in $\nu^a P_w$.  Otherwise, when $|a|> 1$,  $D_{a}=\emptyset$.
  
  Every coset $u P_w$ can be written $\nu^s P_w$ for some $s\in \R$, and
  $$Q_w(\zero,r)\cap \nu^s P_w=\delta_r(R_w \cap \nu^{\frac{s}{r}} P_w)=\delta_r(D_{\frac{s}{r}}).$$
  This is either a quasiball of radius $\approx r$ or the empty set.
  
  Suppose $g\in \H_n$ and $s\in \R$.  Then
  \begin{equation*}
 \begin{split}
  Q_w(g,r)\cap g\nu^s P_w &=
  \Pi_w(g Q_w(\zero,r) \cap g\nu^s \Span(P_w,w)) \\ &=
  \Pi_w(g \delta_r(D_{\frac{s}{r}})) = \delta_r(\Pi_w(\delta_{r^{-1}}(g) D_{\frac{s}{r}})).
\end{split}
  \end{equation*}
  
  We claim that $\Pi_w(\delta_{r^{-1}}(g) D_{\frac{s}{r}})$ is isometric to $D_{\frac{s}{r}}$.  It suffices to consider the case $r=1$.
  
  Let $B=D_{s}$ and let $\alpha=-y_n(g)$.  Since $B\subset V_0$, we have $y_n(gb)=y_n(g)$ for all $b\in B$, so $\Pi_w(gb)=gbw^{\alpha}$.  Since $B\subset \nu^s P_w=s\nu+P_w$ and $\overline{\Omega}(w,P_w)=0$, we have $\overline{\Omega}(b, w)= \overline{\Omega}(s\nu, w)$.  That is, $[b,w]$ is independent of the choice of $b$.
  Therefore,
  \begin{align*}
    \Pi_w(g b) = g b w^\alpha = g [b,w^\alpha] w^\alpha b = g Z^{\alpha \Omega(b, w)} w^\alpha b = g Z^{s \alpha \Omega(\nu, w)} w^\alpha b.
  \end{align*}
  Therefore, $\Pi_w(gB) = g Z^{s\alpha\Omega(\nu, w)} w^\alpha B$ is a left-translate of $B$, so $\Pi_w(gB)$ is a quasiball  of radius $\approx r$ in $g\nu^s P_w$.
\end{proof}

Let $\Gamma$ be an intrinsic Lipschitz graph. When $g \in \Gamma$ is large, $Q_w(g,r)$ can be highly distorted, but the following lemma shows that $Q_w(g,r)$ is the projection of a quasiball with respect to the induced metric on $\Gamma$. Using this property, we'll show that the pushforward measure of $\cH^{2n+1}|_\Gamma$ by $\Pi_w$ is globally equivalent to the Lebesgue measure on $V_0$.

\begin{lemma} \label{l:quasi-Q}
Let $0<\lambda<\lambda'<1$.  There is a $c>0$ depending on $\lambda$ and $\lambda'$ such that for any intrinsic $\lambda$--Lipschitz graph $\Gamma$ such that $\Pi_{Y_n}(\Gamma)=V_0$, any $g\in \Gamma$, any $r>0$, and any horizontal vector $w\in A\cap \Cone_{\lambda'}$ such that $y_n(w)=1$:
\begin{enumerate}
\item \label{l:quasi-Q1} There is a function $f_w\from V_0\to \R$ such that $\Gamma=\Gamma_{f_w,w}$.
\item \label{l:quasi-Q2}
$Q_w(g,c^{-1}r)\subset \Pi_w(B(g,r)\cap \Gamma)\subset Q_w(g,cr).$
\item \label{l:quasi-Q3} Let $\mu$ be Lebesgue measure on $V_0$.  Then $\mu\approx (\Pi_w)_*(\cH^{2n+1}|_\Gamma)$.
\end{enumerate}
\end{lemma}

\begin{proof}
We start with the proof of \eqref{l:quasi-Q1}.  Let $L=(\frac{1}{\lambda}-\frac{1}{\lambda'})^{-1}>0$.  Since $\Pi_{Y_n}(\Gamma)=V_0$, the complement $\H\setminus \Gamma$ has two connected components, $\Gamma^+$ and $\Gamma^-$.  Since $\zero\in \Gamma$, the double cone $\Cone_\lambda$ is disjoint from $\Gamma$, and we can label its two connected components $\Cone^+_\lambda$ and $\Cone^-_\lambda$  so that $\Cone^\pm_\lambda\subset \Gamma^\pm$.

Let $v\in V_0$.  We claim that there is a unique $t_0(v)=t_0\in \R$ such that $vw^{t_0}\in \Gamma$ and that $|t_0|\le Ld(\zero,v)$.  Suppose that $t> L d(\zero,v)$ and let $h = v w^t$. Note that $h\in \Cone^+_\lambda$. Indeed,  since $w\in \Cone_{\lambda'}$, we have $|y_n(w)|>\lambda' d(\zero,w)=\lambda' \|w\|$, i.e., $\|w\|< (\lambda')^{-1}$.  By the triangle inequality,
$$d(\zero, h) - \frac{y_n(h)}{\lambda} \le d(\zero,v) + t \cdot \|w\| - \frac{t}{\lambda} \le d(\zero,v)+ \frac{t}{\lambda'} - \frac{t}{\lambda} = d(\zero,v)-Lt < 0,$$
so $d(\zero, h)< \frac{y_n(h)}{\lambda}$ and thus $vw^t\in \Cone^+_\lambda$.  Likewise, $vw^{-t}\in \Cone^-_\lambda$.  These points are on different sides of $\Gamma$, so there is some $t_0$ such that $v w^{t_0}\in \Gamma$; in fact, we can take $|t_0|\le L d(\zero,v)$.  Since $w^\alpha\in \Cone_\lambda$ for any $\alpha\ne 0$, this $t_0$ is unique.  Therefore, $\Gamma$ is a $w$--intrinsic graph with $\Pi_w(\Gamma)=V_0$, and we define $f_w\from V_0\to \R$ to be the function such that $v w^{f_w(v)}\in \Gamma$ for all $v\in V_0$. In particular, $f_w(v)=t_0(v)$.

For the proof of \eqref{l:quasi-Q2}, note that by scaling, we may assume $r = 1$. We first prove \eqref{l:quasi-Q2} for $g = \zero$ (so we are assuming $\zero \in \Gamma$). If $p\in B(\zero,1)$, then $|y_n(p)|\le 1$, so 
$$d(\zero,\Pi_w(p)) \le d(\zero,p)+d(p,\Pi_w(p)) \le 1 + |y_n(p)|\cdot \|w\| \lesssim_{\lambda'} 1.$$
Therefore, there is a $C>0$ depending on $\lambda'$ such that $\Pi_w(B(\zero,1))\subset V_0\cap B(\zero,C)$, and if $c$ is sufficiently large (depending only on $\lambda'$), then $\Pi_w(B(\zero,1))\subset Q_w(\zero,c)$.  This proves one inclusion.

We now prove the other inclusion. Note first, that
$$d(\zero,v w^{f_w(v)})\leq d(\zero,v)+\|w\||f_{w}(v)| \leq  \left(1+\frac{L}{\lambda'}\right) d(\zero,v)$$
and consequently
$v w^{f_w(v)}\in \Gamma\cap B(\zero, D d(\zero,v))$ for all $v \in V_0$, where $D=1+\frac{L}{\lambda'}$.  Thus, if $d(\zero,v)<D^{-1}$, then $v\in \Pi_w(B(\zero, 1)\cap \Gamma)$.  If $c$ is sufficiently large (depending only on $\lambda, \lambda'$), then 
$$Q_w(\zero,c^{-1})\subset V_0 \cap B(\zero,D^{-1})\subset \Pi_w(B(\zero, 1)\cap \Gamma),$$
as desired.

Now assume $g \neq \zero$.  Then $g^{-1}\Gamma$ is still an intrinsic Lipschitz graph which now contains $\zero$, so
\begin{align}
    Q_w(\zero,c^{-1}) \subset \Pi_w(B(\zero,1) \cap g^{-1}\Gamma) \subset Q_w(\zero,c). \label{e:0g-contain}
\end{align}

Let $h' \in B(g,1) \cap \Gamma$.  Then $h' = gh$ for some $h \in B(\zero,1) \cap g^{-1}\Gamma$.  Let $y=\Pi_w(h)$; by the upper bound of \eqref{e:0g-contain}, $y \in Q_w(\zero,c) = \delta_c(R_w)$ and $y=hw^\alpha$ for some $\alpha\in \R$.  Then
$gy\langle w\rangle=gh \langle w\rangle=h'\langle w\rangle$, so $\Pi_w(h')=\Pi_w(gy)$ and thus $\Pi_w(h')\in \Pi_w(g\delta_c(R_w))=Q_w(g,c)$. This proves the upper bound of part \ref{l:quasi-Q2}, and the lower bound follows similarly.

To prove the last part, let $\Psi\from V_0\to \Gamma$ be the map $\Psi(v)=v w^{f_w(v)}$, so that $\Psi$ is inverse to $\Pi_w|_\Gamma$.  By \eqref{eq:measure of Q} and part~\ref{l:quasi-Q2}, $\mu(\Pi_w(B(x,r)\cap \Gamma))\approx r^{2n+1}$ for any $x\in \Gamma$.  Thus, by Theorem~2.4.3 of \cite{AT04}, $\Psi_*(\mu)|_\Gamma \approx \cH^{2n+1}|_\Gamma$ and thus $\mu \approx (\Pi_w)_*(\cH^{2n+1}|_\Gamma)$. 
\end{proof}

For every $k\le n$, we identify $\H_k$ with $\Span ( X_1,\dots, X_{k}, Y_1,\dots, Y_{k}, Z )\subset \H_n$.  The vertical subspace $\{y_n=0\}$ is then the internal direct product of $\H_{n-1}$ and $\langle X_n\rangle$.  In fact, every codimension--1 vertical subspace $P\subset \H_n$ is isomorphic to $\H_{n-1}\times \R$, so the cosets of $P$ decompose $\H_n$ into copies of $\H_{n-1}\times \R$.  In the following, we will use these decompositions to construct a family of decompositions of an intrinsic Lipschitz graph $\Gamma$ into a union of graphs of Lipschitz functions $\H_{n-1}\to \R$.

Recall that for every $w\in A$, we defined $P_w=V_0\cap w^{\overline{\Omega}}$; this is a codimension--1 subspace of $V_0$.
Then $P_w\cap \langle w \rangle=\{\mathbf{0}\}$, $P_w\cdot \langle w \rangle=w^{\overline{\Omega}}$, and $w$ and $P_w$ commute, so $w^{\overline{\Omega}}\cong \langle w\rangle \times P_w.$
Furthermore, since $\pi(P_w)=\Span(X_n,w)^\Omega$ and $\Span(X_n,w)$ is a symplectic subspace of $A$, $\pi(P_w)$ is a symplectic subspace and thus $P_w$ is isomorphic to $\H_{n-1}$.

By Lemma~\ref{l:quasi-Q}, if $\Gamma$ is an intrinsic Lipschitz graph, then for every horizontal vector $w\in A$ that is sufficiently close to $Y_n$, there is a function $f_w$ such that $\Gamma=\Gamma_{f_w,w}$.  These functions satisfy a Lipschitz condition on cosets of $P_w$.
\begin{lemma}\label{lem:nearby slices}
  Let $0<\lambda<\lambda'<1$.  There is an $L>0$ such that if \(\Gamma\) is an intrinsic $\lambda$--Lipschitz graph, $w\in A\cap \Cone_{\lambda'}$, $y_n(w)=1$, and $f_w\from V_0 \to \R$ is the parametrizing function of $\Gamma$, i.e.,  \(\Gamma=\Gamma_{f_w,w}\), then for any $g\in V_0$, the restriction $f_w|_{gP_w}$ is $L$--Lipschitz with respect to $d_{\H_n}$.  
\end{lemma}

\begin{proof}
  Let $f_w\from V_0\to \R$ be the function such that $\Gamma=\Gamma_{f_w,w}$, as in Lemma~\ref{l:quasi-Q} and let $g\in \H_n$.  We claim that $f_w|_{gP_w}$ is Lipschitz.  We may assume $g = \zero$ as $d$ is left-invariant.  Define $\Psi\from \Pi_w(\Gamma)\to \Gamma$, $\Psi(u)=uw^{f_w(u)}$.

  Let $u,v \in P_w\cap \Pi_w(\Gamma)$.  Since $\Psi(u),\Psi(v)\in \Gamma$, the intrinsic Lipschitz condition implies that $\Psi(u) \notin \Psi(v) \cdot \Cone_\lambda$ or
  \begin{align*}
    h := \Psi(v)^{-1} \Psi(u) \notin \Cone_\lambda.
  \end{align*}
  Since $v^{-1}u \in P_w$ is in the $\overline{\Omega}$--complement of $w$, it commutes with $w$, so we can decompose $h$ as
  \begin{align*}
    h = (vw^{f_w(v)})^{-1}(uw^{f_w(u)}) = w^{-f_w(v)}v^{-1}u w^{f_w(u)} = v^{-1}u \cdot w^{f_w(u)-f_w(v)}.
  \end{align*}
  Then, since $v^{-1}u\in V_0$,
  \begin{equation}
  \label{lem:nearby slices-1}
    |y_n(h)| = |y_n(w^{f_w(u) - f_w(v)})| = |f_w(u)-f_w(v)|.
  \end{equation}
  Moreover, since $w \in \Cone_{\lambda'}$ and $y_n(w)=1$,
    \begin{equation}
    \begin{split}
  \label{lem:nearby slices-2}
    d(h,\mathbf{0}) &\le d(u,v)+d(w^{f_w(u)}, w^{f_w(v)})\\ 
    & = d(u,v) + d(\mathbf{0},w) |f_w(v) - f_w(u)|\\
    & \le d(u,v) + \frac{1}{\lambda'} |f_w(v) - f_w(u)|.
  \end{split}
  \end{equation}
  
  Since $h \notin \Cone_\lambda$, we have $|y_n(h)|\le \lambda d(\mathbf{0},h)$.  Then,
  \begin{equation*}
 |f_w(u)-f_w(v)|\overset{\eqref{lem:nearby slices-1}}{=} |y_n(h)| \leq \lambda d(\mathbf{0},h)  \overset{\eqref{lem:nearby slices-2}}{\le} \lambda \left(d(u,v) + \frac{1}{\lambda'} |f_w(v) - f_w(u)|\right)
  \end{equation*}
 and  by the fact that $\lambda'>\lambda$,
  \begin{equation*}
  |f_w(u)-f_w(v)| \le \frac{\lambda}{1-\frac{\lambda}{\lambda'}} d(u,v) =
 \frac{\lambda\lambda'}{\lambda'-\lambda}d(u,v). 
  \end{equation*}
\end{proof}

\section{Comparing $\beta$--numbers}\label{sec:comparing}

Since Theorem~\ref{thm:beta2 highdim} deals with intrinsic graphs, its proof will use parametric versions of $\beta$--numbers, i.e., quantities that measure how close a function $f$ is to an affine function rather than how close the graph $\Gamma_f$ is to a plane.  In Lemma~\ref{lem:nearby slices}, we saw that there are many different ways of writing $\Gamma$ as a graph, all of which lead to different parametric $\beta$--numbers.  In this section, we prove some inequalities comparing parametric and non-parametric $\beta$--numbers.

We first prove the following simple lemma.
\begin{lemma}\label{lem:beta alpha}
  Let $0<\lambda<\lambda'<1$.  There is a $c>0$ with the following property.  Let $w\in A\cap \Cone_{\lambda'}$ be such that $y_n(w)=1$.  Let $f \from V_0\to \R$ be an intrinsic $(w,\lambda)$--Lipschitz function and let $\Gamma:=\Gamma_{f,w}$ be its intrinsic graph.  For any $r>0$ and any $x\in \Gamma$, 
  \begin{equation}\label{eq:beta alpha}
  \beta_{\Gamma}(x,r) \lesssim_{\lambda,\lambda'} r^{-\frac{2n+3}{2}} \inf_{h \in \Aff} \|f - h\|_{L_2(Q_{w}(x,cr))}.
  \end{equation}
\end{lemma}

\begin{proof}
  Let $c$ be as in Lemma~\ref{l:quasi-Q} so that  $\Pi_w(B(x,r) \cap \Gamma)\subset Q_w(x,cr)$.  Let $Q=Q_w(x,cr)$.
  Let $h\in \Aff$ be the affine function that minimizes $\|h-f\|_{L_2(Q)}.$
  Let $S=\Gamma_{h,w}$ and let $\mu$ be Lebesgue measure on $V_0$.
  Then $S$ is a vertical plane, so 
  \begin{align*}
  \beta_{\Gamma}(x,r)^2 & \le r^{-2n-3} \int_{B(x,r)\cap \Gamma_f} d(y,S)^2 \ud \cH^{2n+1}(y) \\ & \stackrel{Lem.~\ref{l:quasi-Q}}{\approx_{\lambda,\lambda'}}r^{-2n-3}\int_{\Pi_w(B(x,r)\cap \Gamma_f)} d(vw^{f(v)},S)^2 \ud \mu(v) \\
  & \le r^{-2n-3}\int_{Q} d(vw^{f(v)},vw^{h(v)})^2 \ud \mu(v)\\  & \lesssim_{\lambda,\lambda'} r^{-2n-3}\|f-h\|_{L_2(Q)}^2.
  \end{align*}
  Taking square roots of both sides gives the desired inequality.
\end{proof}

Let $P\subset V_0$ be a vertical subgroup.  We say that a measurable function $f$ is \emph{$P$--slice affine} if for every $v\in V_0$, the restriction of $f$ to $vP$ is a vertical affine function.  Let $\SAff_P$ be the set of $P$--slice affine functions. Note that we have the following inclusion of function spaces:
\begin{align}
    \Aff \subseteq \SAff_P \subseteq \Vert. \label{e:fcn-inclusion}
\end{align}

The second result we will need deals with different parameterizing functions for the same graph.   Suppose that $\Gamma$ is an intrinsic Lipschitz graph, $w$ and $w'$ are horizontal vectors with $y_n(w)=y_n(w')=1$ and $[w,w']=0$. Moreover, assume that $f$ and $f'$ are two parameterizing functions for $\Gamma$, so that $\Gamma=\Gamma_{f,w}=\Gamma_{f',w'}$.  The relationship between $f$ and $f'$ is nonlinear.  Indeed, by Lemma~\ref{lem:nearby slices}, $f$ is Lipschitz on cosets of $P=P_w=V_0\cap w^{\overline{\Omega}}$, while $f'$ need not be.  Regardless, we will show that if $f$ is close to $P$--slice affine, then $f'$ is close to $P$--slice affine too.  

\begin{lemma}\label{lem:alpha compare saff}
  Let $0<\lambda<1$.  There exist $c>0$ and $\lambda' \in (\frac{1+\lambda}{2},1)$ depending only on $\lambda$ so that the following holds.  Let $w,w' \in A \cap \Cone_{\lambda'}$ be horizontal vectors with $y_n(w)=y_n(w')=1$ and $[w,w']=0$. 
  
  Let $\Gamma$ be a $\lambda$--intrinsic Lipschitz graph with parameterizing functions $f, f' \from V_0 \to \R$ such that $\Gamma=\Gamma_{f,w}=\Gamma_{f',w'}$.  Then, for any $x \in \Gamma$ and $r > 0$,
  \begin{equation}\label{eq:alpha compare saff}
    \min_{\sigma\in \SAff_{P_w}} \|f'-\sigma\|_{L_2(Q_{w'}(x,r))} \lesssim \min_{\sigma\in \SAff_{P_w}} \|f-\sigma\|_{L_2(Q_w(x,cr))}.
  \end{equation}
\end{lemma}

The key to this bound is that since $[w,w']=0$, we have $w-w'\in P_w$ and thus $P_w+\langle w\rangle=P_w+\langle w'\rangle$.  Thus, if $F\in \SAff_{P_w}$ satisfies a Lipschitz bound on each slice, then there is an $F'\in \SAff_{P_w}$ such that $\Gamma_{F,w}=\Gamma_{F',w'}$.

We will need two lemmas dealing with affine functions on slices.
\begin{lemma} \label{l:lip-bound}
  Let $\mu>0$.  There is a constant $c > 1$ depending only on $\mu$ and $n$ so that if $f \from \H_n \to \R$ is a Lipschitz function and $U \subset \H_n$ is a $\mu$--quasiball, then the affine function $F\from \H_n \to \R$ such that
  \begin{align*}
    \|f - F\|_{L_2(U)} = \inf_{T \in \Aff} \|f- T\|_{L_2(U)}
  \end{align*}
  satisfies $\|F\|_{\lip} \le c \|f\|_{\lip}$, where the infimum is taken over affine functions on $\H_n$.
\end{lemma}

\begin{proof}
  Since $U$ is a $\mu$--quasiball, there is a $u_0\in U$ and an $r>0$ such that 
  $$B(u_0, r)\subset U \subset B(u_0,\mu r).$$
  By rescaling and translation, we may suppose that $r=1$ and $u_0=\zero$.
  Let $\langle g,h\rangle_U=\int_U gh\ud \cH^{2n+2}$ be the inner product on $L_2(U)$ and let $\|\cdot\|_2$ denote the norm in $L_2(U)$.  Let $f_0=f-f(\zero)$, and let $F_0=F-f(\zero)$.  Then $$\|f_0 - F_0\|_{2} = \inf_{T \in \Aff} \|f_0- T\|_{2},$$
  so it suffices to prove the lemma for $f_0$ and $F_0$.
  
  Let 
  $$C_1=\max_{K\in \Aff} \frac{\|K\|_{\lip}}{\|K\|_{L_2(B(\zero,1))}}=\max_{K\in \Aff} \frac{\|K\|_{\lip}}{\|K-\fint_{B(\zero,1)} K\ud \cH^{2n+2}\|_{L_2(B(\zero,1))}}.$$
  This is finite by compactness, and if $K\in \Aff$, then 
  $$\|K\|_{\lip} \le C_1 \|K\|_{L_2(B(\zero,1))} \le C_1 \|K\|_{2}.$$
  Let
  $$C_2=\max_{\substack{\|g\|_{\lip} < \infty \\ g(\zero)=0}} \frac{\|g\|_{2}}{\|g\|_{\lip}}.$$
  If $g(\zero)=0$, then $\|g\|_{L_\infty(U)}\le \mu \|g\|_{\lip}$, so 
  $$C_2\le \mu \sqrt{\cH^{2n+2}(U)} \lesssim \mu^{n+2}.$$

  Since $f_0-F_0$ is orthogonal to $F_0$, we have $\|F_0\|_2\le \|f_0\|_2$ and thus
  \begin{equation*}
  \|F_0\|_{\lip} \le C_1 \|F_0\|_2\le C_1 \|f_0\|_2\le C_1C_2\|f_0\|_{\lip}\lesssim \mu^{n+2} \|f_0\|_{\lip}.\qedhere
  \end{equation*}
\end{proof}

\begin{lemma}\label{l:switch affine}
  Let $L>0$.  There is a $\lambda\in (0,1)$ such that if $w,w' \in A \cap \Cone_{\lambda}$, $y_n(w)=y_n(w')=1$, $[w,w']=0$, and $T\from P_w\to \R$ is an affine function with $\Lip(T)<L$, then there is an affine function $T'\from P_w\to \R$ such that $\Gamma_{T,w}=\Gamma_{T',w'}$ and $\Lip(T')<2L$.
\end{lemma}
\begin{proof}
  Let $C_w=\pi(P_w)$, which we think of as a vector space with norm $\|\cdot\|$.
  As affine functions do not depend on the $z$--coordinate, we see that $T = \tau \circ \pi$ where $\tau\from C_w \to \R$ is affine and $\Lip(T)=\Lip(\tau)$.
  Let $\lambda\in (0,1)$ be such that if $w,w'\in A\cap \Cone_\lambda$ and $y_n(w)=y_n(w')=1$, then  $\|w-w'\|<(2L)^{-1}$.  Let $s=w-w'$.

  Since $[w,s]=-[w,w']=0$ and $y_n(s)=0$, we have $s\in P_w$.  
  Let $M\from P_w\to P_w$ be the map 
  $$M(p)=\Pi_{w'}(pw^{T(p)})=pw^{T(p)}(w')^{-T(p)}=ps^{T(p)} \qquad \forall p\in P_w.$$
  
  Since $M(v+tZ)=M(v)+tZ$ for all $v\in P_w$ and $t\in \R$, $M$ descends to a map $m\from C_w\to C_w$.
  Let $m=\pi\circ M|_{C_w}$.  By our choice of $\lambda$, for all $p,q\in C_w$,
  $$\|m(p)-m(q)\|=\|p-q+(\tau(p)-\tau(q))s\|\ge \|p-q\|-\Lip(\tau) \|p-q\| \cdot \|s\| \ge \frac{\|p-q\|}{2}.$$
  Therefore $m$ and $M$ are invertible and $\Lip(m^{-1})\le 2$. 
  
  Let $T'(M(p))=T(p)$ for all $p\in P_w$.  This is affine and by construction,
  $$p w^{T(p)} = M(p) (w')^{T(p)}=M(p)(w')^{T'(M(p))},$$
  so $\Gamma_{T,w}=\Gamma_{T',w'}$.  Finally, $T'=\tau\circ m^{-1}\circ \pi$, so $\Lip(T')=\Lip(\tau\circ m^{-1}) < 2L$.  
\end{proof}


Now, we prove Lemma~\ref{lem:alpha compare saff}.
\begin{proof}[{Proof of Lemma~\ref{lem:alpha compare saff}.}]
  By Lemma~\ref{l:quasi-Q}, there is a $c_0>0$ such that for every $x\in \Gamma$ and every $r>0$, $Q_{w'}(x,r)\subset \Pi_{w'}(B(x,c_0r)\cap \Gamma)$ and $\Pi_{w}(B(x, r)\cap \Gamma)\subset Q_w(x,c_0r)$.  Let $c=c_0^2$.
  
  Let $P=P_w$.
  We first consider a single slice $Q_{w}(x,cr) \cap vP$; without loss of generality, we may take $v=\zero$. Let $D := Q_{w}(x,cr) \cap P$ and $D' := Q_{w'}(x,r) \cap P$.  Let $T\from P\to \R$ be the affine function minimizing $\|f-T\|_{L_2(D)}$.    We know from Lemma \ref{lem:nearby slices} that $f$ is Lipschitz on $P$, and if $\frac{1+\lambda}{2}<\lambda'<1$, then $\Lip(f|_P)$ is bounded by a function of $\lambda$.  Therefore, by Lemma~\ref{l:lip-bound}, $\Lip(T)$ is also bounded by a function of $\lambda$, say $\Lip(T)<L=L(\lambda)$.
  
  The intrinsic graph $\Gamma_{T,w}$ is a plane which we call $R$.  We suppose that $\lambda'\in (0,1)$ is close enough to $1$ that we may apply Lemma~\ref{l:switch affine} to find an affine function $T'\from P\to \R$ such that $R=\Gamma_{T',w'}$ and $\Lip(T')\le 2L$.

  Let $g\in \Gamma\cap (P+\langle w\rangle)$ and $u=\Pi_w(g)$.  Then $g=u w^{f(u)}$ and thus $f(u)=y_n(g)$.  Let $h=g w^{T(u)-f(u)}$; then $h=uw^{T(u)}\in R$.
  Let $u'=\Pi_{w'}(g)$ and $v'=\Pi_{w'}(h)$, so that $f(u)=f'(u')=y_n(g)$ and $T(u)=T'(v')=y_n(h)$.  Since $[w,w']=0$, we let $s=w-w'$ and write $$u' = uw^{f(u)}(w')^{-f(u)}=us^{f(u)}$$ and $v'=us^{T(u)}$.
  
  We have
  \begin{multline*}
  |f'(u')-T'(u')|\le |f(u)-T(u)|+|T(u)-T'(u')| \\ =  |f(u)-T(u)|+|T'(v')-T'(u')| 
   \le  |f(u)-T(u)|+\Lip(T')d(u',v').
  \end{multline*}
  Since $d(u',v') = \|s\|\cdot |f(u)-T(u)|,$ 
  \begin{equation}\label{eq:f'-T'}
  |f'(\Pi_{w'}(g))-T'(\Pi_{w'}(g))| \lesssim |f(\Pi_{w}(g))-T(\Pi_{w}(g))|
  \end{equation}
  for every $g\in \Gamma \cap (P+\langle w\rangle)$.
  
  Applying the argument above to every coset of $P$ shows that if $\sigma\in \SAff_P$ minimizes $\|f-\sigma\|_{L_2(Q_w(x,cr))}$, then there is a slice-affine function $\sigma'\in \SAff_P$ such that
  \begin{equation}\label{eq:f'-sigma'}
  |f'(\Pi_{w'}(g))-\sigma'(\Pi_{w'}(g))| \lesssim |f(\Pi_{w}(g))-\sigma(\Pi_{w}(g))|
  \end{equation}
  for every $g\in \Gamma$.
  
  By Lemma~\ref{l:quasi-Q}, if $\mu$ is Lebesgue measure on $V_0$, then 
  \begin{equation}\label{eq:measure equiv}\mu\approx (\Pi_w)_*(\cH^{2n+1}|_{\Gamma}) \approx (\Pi_{Y_n})_*(\cH^{2n+1}|_{\Gamma}).
  \end{equation}
  Therefore,
  \begin{equation*}
      \begin{split}
  \|f'-\sigma'\|^2_{L_2(Q_{w'}(x,r))} &\overset{\eqref{eq:f'-sigma'}}{\lesssim} \int_{\Gamma\cap B(x,\sqrt{c} r)} \left(f'(\Pi_{w'}(g))-\sigma'(\Pi_{w'}(g))\right)^2 \ud \cH^{2n+1}(g) \\ &\overset{\eqref{eq:measure equiv}}{\lesssim} \int_{\Gamma\cap B(x,\sqrt{c} r)} \left(f(\Pi_{w}(g))-\sigma(\Pi_{w}(g))\right)^2 \ud \cH^{2n+1}(g) \\
  &\overset{\eqref{eq:f'-sigma'}}{\lesssim} \|f-\sigma\|^2_{L_2(Q_w(x,cr))}.
  \end{split}
  \end{equation*}
  Thus 
  $$\|f'-\sigma'\|_{L_2(Q_{w'}(x,r))} \lesssim \|f-\sigma\|_{L_2(Q_w(x,cr))} = \min_{\sigma \in \SAff_P} \|f-\sigma\|_{L_2(Q_w(x,cr))},$$
  as desired.
  
\end{proof}

\section{Strong geometric lemma for intrinsic Lipschitz graphs}\label{sec:proof of main}
In this section, we will prove Theorem~\ref{thm:beta2 highdim}, modulo a bound that we will prove in Section~\ref{sec:slice affine}.  Our strategy is to reduce from the intrinsic Lipschitz graph $\Gamma$ to Lipschitz graphs on cosets of $\H_{n-1}$ by slicing.  By the results of Section~\ref{sec:slices}, there are codimension--1 vertical subgroups $W\subset \H_n$ such that the intersections $vW\cap \Gamma$ are graphs of Lipschitz functions defined on subgroups isomorphic to $\H_{n-1}$.

We start with some standard notation. Let $f\from \H_n \to \R$ be a smooth function.  The \emph{horizontal derivatives} of $f$ in the directions $X_i, Y_i, i=1,\dots,n,$ are the left invariant vector fields:
$$X_i f(h):=\frac{\partial f}{\partial x_i} (h)-\frac{1}{2}y_i(h) \frac{\partial f}{\partial z} (h) \mbox{ and }Y_i f(h):=\frac{\partial f}{\partial y_i} (h)+\frac{1}{2}x_i(h) \frac{\partial f}{\partial z} (h), \quad h\in \H_n.$$
The \emph{horizontal gradient} of $f$ is defined as
$$\nabla_H f (h)= (X_1f(h), \dots, X_nf(h), Y_1f(h)\dots Y_nf(h)), \quad h \in \H_n.$$

F\"{a}ssler and Orponen proved that Lipschitz functions on Heisenberg groups satisfy the following version of Dorronsoro's Theorem.  For any function $f\from \H_n \to \R$ and any $x\in \H_n$, $r>0$, let 
\begin{align}\label{eq:define theta}
  \theta_f(B(x,r)) = r^{-2n-4} \inf_{g \in \Aff} \|f - g\|_{L_2(B(x,r))}^2 \approx \inf_{g\in \Aff} \fint_{B(x,r)} \frac{(f(x)-g(x))^2}{r^2}\ud x.
\end{align}
\begin{thm} \label{th:FO} \cite[Theorem 6.1]{FO}
  Let $f \in L_2(\H_n)$ with $\nabla_H f \in L_2(\H_n)$.  Then
  \begin{align*}
    \int_{\H_n} \int_0^\infty \theta_f(B(x,r)) \frac{dr}{r} ~dx \lesssim \|\nabla_H f\|_{L_2(\H_n)}^2.
  \end{align*}
\end{thm}

This will let us bound how well each slice $vW\cap \Gamma$ can be approximated by planes, but even if $vW \cap \Gamma$ is close to a plane for every $v$, the whole graph $\Gamma$ may not be.
We thus prove Theorem~\ref{thm:beta2 highdim} by considering slices of $\Gamma$ parallel to several different planes $W_i$.  
Using a bound that will be proved in Section~\ref{sec:slice affine}, we will show that if $vW_i \cap \Gamma$ is close to a plane for many different choices of $W_i$, then the whole graph $\Gamma$ is close to a plane.  We then deduce the strong geometric lemma for $\Gamma$ by applying Theorem~\ref{th:FO} to each slice.

We will need a weak local version of Theorem~\ref{th:FO}.
\begin{cor} \label{c:FO}
  Let $f \from \H_n \to \R$ be a Lipschitz function and $B(y,R) \subset \H_n$ a ball.  Then
  \begin{align*}
    \int_{B(y,R)} \int_0^R \theta_f(B(x,r)) \frac{\ud r}{r} \ud x \lesssim \|f\|_{\textrm{Lip}}^2 R^{2n+2}.
  \end{align*}
\end{cor}

\begin{proof}
  We translate so that $y=\zero$ and $f(\zero)=0$.  We first construct a Lipschitz function $F \from \H_n \to \R$ such that $F=f$ on $B(\zero,2R)$ and $\|\nabla_H F\|_{L_2(H_n)^2} \lesssim \|f\|_{\textrm{Lip}}^2 R^{2n+2}$.

  Let $S=B(\zero,2R)\cup (\H\setminus B(\zero,4R))$ and define $F\from S\to \R$ by
  \begin{align*}
    F(x) = \begin{cases}
      f(x), & d(\zero,x) \le 2R, \\
      0  & d(\zero,x) \ge 4R.
    \end{cases}
  \end{align*}
  If $p\in B(\zero,2R)$ and $d(\zero,q)\ge 4R$, then 
  $$\frac{|f(p)-f(q)|}{d(p,q)}\le \frac{\|f\|_{L_\infty(B(\zero,2R))}}{2R} \le \Lip(f),$$
  so $\Lip(F)=\Lip(f)$.  Since $\R$ is an absolute $1$--Lipschitz retract \cite[Ch.\ 1]{BL00}, we can extend $F$ to a Lipschitz function from $\H$ to $\R$ with $\Lip(F) = \Lip(f)$.  This function is supported in $B(\zero,4R)$, so
  $$\|\nabla_H F\|_{L_2(\H_n)}^2 \leq \Lip(F)^2 \cH^{2n+2}(B(\zero,4R))\lesssim \Lip(f)^2 R^{2n+2}$$
  as desired.

  If $x \in B(\zero,R)$ and $r < R$, then $f|_{B(x,r)} = F|_{B(x,r)}$ and thus $\theta_F(B(x,r)) = \theta_f(B(x,r))$.  Then
  \begin{align*}
    \int_{B(y,R)} \int_0^R \theta_f(B(x,r)) \frac{dr}{r} ~dx &= \int_{B(y,R)} \int_0^R \theta_F(B(x,r)) \frac{dr}{r} ~dx  \\
    &\leq \int_{\H_n} \int_0^\infty \theta_F(B(x,r)) \frac{dr}{r} ~dx  \\
    &\lesssim \|\nabla_H F\|_{L_2(\H_n)}^2 \\
    &\leq \|f\|_{\textrm{Lip}}^2 R^{2n+2}.
  \end{align*}
\end{proof}

We now prove the main theorem.

\begin{proof}[Proof of Theorem~\ref{thm:beta2 highdim}]
Let $Q(x,r)=Q_{Y_n}(x,r)$.
Let $\lambda'\in (\frac{1+\lambda}{2},1)$ be as in Lemma~\ref{lem:alpha compare saff}.  

We first  construct some $(2n-1)$--dimensional vertical subspaces $P,P_1,\dots, P_{2n-1}$.  Recall that $P_w=V_0\cap w^{\overline{\Omega}}$ and let
$$P:=P_{Y_n}=\Span (X_1,\dots, X_{n-1}, Y_1,\dots, Y_{n-1},Z).$$
Let $W_0=P+\langle Y_n\rangle= \{x_n=0\}$ and let $w_1,\dots,  w_{2n-1}\in A\cap W_0 \cap \Cone_{\lambda'}$ be linearly independent vectors such that $y_n(w_i)=1$.  Let $P_i:=P_{w_i}=V_0\cap w_i^{\overline{\Omega}}$.  Each vector $w_i$ is close to $Y_n$, so each plane $P_i$ is close to $P$.  Since the $w_i$'s span $A\cap W_0$, we have
$$\bigcap_i w_i^{\overline{\Omega}}= (A\cap W_0)^{\overline{\Omega}}=\Span(Z,Y_n),$$
and $\bigcap_i P_i=V_0\cap\bigcap_i w_i^{\overline{\Omega}}=\langle Z\rangle$. 
For all $i$ and all $x\in \Gamma$, let $Q_i(x,r) := Q_{w_i}(x,r)$.
 
In Lemma~\ref{lem:slicing vertical} in Section~\ref{sec:slice affine}, we will show that there is a $c>1$ depending on $\lambda$ and the $P_i$'s such that if $f^v\from V_0\to \R$ is a vertical function, then
$$\min_{h \in \Aff} \|f^v - h\|_{L_2(Q(x,r))} \lesssim \sum_i \min_{\sigma \in \SAff_{P_i}} \|f^v - \sigma\|_{L_2(Q(x,cr))}$$
for any $x\in \Gamma$ and $r>0$.  We also suppose that $c$ is large enough to satisfy Lemma~\ref{lem:alpha compare saff}, Lemma~\ref{l:quasi-Q}, and Lemma~\ref{lem:beta alpha}.  In particular,
\begin{align*}
  r^{n+3/2} \beta_\Gamma(x,r) &{\lesssim} \min_{h \in \Aff} \|f - h\|_{L_2(Q(x,cr))} = (*)
\end{align*}
for any $x\in \Gamma$.

Now, fix some $x\in \Gamma$ and let $f^v\in \Vert$ be the vertical function such that \begin{equation}\label{eq:def fv}
\|f-f^v\|_{L_2(Q(x,cr))}=\min_{h\in \Vert}\|f-h\|_{L_2(Q(x,cr))}.
\end{equation}
Let $f_{w_i}\from V_0\to \R$ be the parameterizing functions such that $\Gamma=\Gamma_{f_{w_i},w_i}$.
Then, by the triangle inequality,
\begin{align*}
  (*) &\leq \|f - f^v\|_{L_2(Q(x,cr))} + \min_{h \in \Aff} \|f^v - h\|_{L_2(Q(x,cr))} \\
  &\stackrel{\text{Lem.~\ref{lem:slicing vertical}}}{\lesssim} \|f - f^v\|_{L_2(Q(x,cr))} + \sum_i \min_{\sigma \in \SAff_{P_i}} \|f^v - \sigma\|_{L_2(Q(x,c^2r))}  \\
  &\lesssim \|f - f^v\|_{L_2(Q(x,cr))} + \sum_i \min_{\sigma \in \SAff_{P_i}} (\|f^v - f\|_{L_2(Q(x,c^2r))} + \|f - \sigma\|_{L_2(Q(x,c^2r))}) \\
  &\stackrel{\text{Lem.~\ref{lem:alpha compare saff}}}\lesssim \|f - f^v\|_{L_2(Q(x,c^2r))} + \sum_i \min_{\sigma \in \SAff_{P_i}}  \|f_{w_i} - \sigma\|_{L_2(Q_i(x,c^2r))}.
\end{align*}
By \eqref{eq:def fv} and the inclusion $\SAff_{P} \subseteq \Vert$, 
\begin{align}
  r^{n+3/2} \beta_\Gamma(x,r) &\leq \|f - f^v\|_{L_2(Q(x,c^2r))} + \sum_i \min_{\sigma \in \SAff_{P_i}} \|f_{w_i} - \sigma\|_{L_2(Q_i(x,c^2r))} \notag \\
  &\leq \min_{\sigma \in \SAff_{P}} \|f - \sigma\|_{L_2(Q(x,c^2r))} + \sum_i \min_{\sigma \in \SAff_{P_i}} \|f_{w_i} - \sigma\|_{L_2(Q_i(x,c^2r))} \label{e:nonpar-par}
\end{align}
Thus, in order to prove Theorem~\ref{thm:beta2 highdim}, it suffices to bound how well $f$ and the $f_{w_i}$'s can be approximated by slice-affine functions.


Let $p_0\in \Gamma$, $R>0$, and $B=B(p_0,R)$.  After a translation, we may suppose $p_0=\zero$.  Let 
$$\gamma(x,r)=r^{-2n-1} \min_{\sigma \in \SAff_P} \frac{\|f - \sigma\|^2_{L_2(Q(x,r))}}{r^2}.$$
We consider
$$I=\int_0^R r^{-2n-1} \int_{B \cap \Gamma} \min_{\sigma \in \SAff_P} \frac{\|f - \sigma\|^2_{L_2(Q(x,c^2r))}}{r^2} \ud x \frac{\ud r}{r} \approx \int_0^R \int_{B \cap \Gamma} \gamma(x,c^2r) \ud x \frac{\ud r}{r}.$$
We first pass from the integral above to a sum over the balls in a family of bounded-multiplicity covers $\cC_k$ of $\Gamma$.

By Lemma~\ref{l:quasi-Q} and our choice of $c$, we have
\begin{equation}\label{eq:c nesting}
Q(x,c^{-1}r)\subset \Pi(B(x,r)\cap \Gamma)\subset Q(x,cr)
\end{equation}
for any $x\in \Gamma$ and $r>0$.  For $k\ge 0$, let $\cN_k$ be a maximal $R2^{-k}$--net in $\Gamma\cap B$.  Let $r_k=2c^3 R2^{-k}$ and let $\cC_k=\{Q(p,r_k): p\in \cN_k\}$.  By the Ahlfors regularity of $\Gamma$, $\cC_k$ has bounded multiplicity, i.e., for any $x\in V_0$, $|\{Q\in \cC_k:x\in Q\}|\lesssim 1$.  Furthermore, there is a $D>0$ such that if $Q\in \cC_k$, then $Q\subset Q(\zero,DR)$.

For any $x\in \Gamma\cap B$, there is an $p\in \cN_k$ such that $d(x,p)\le R2^{-k}$, so by \eqref{eq:c nesting},
$$Q(x,R2^{-k}) \subset \Pi\bigl(B(x,c^2R2^{-k})\cap \Gamma\bigr) \subset \Pi\bigl(B(p,2c^2R2^{-k})\cap \Gamma\bigr) \subset Q(p,2c^3R2^{-k}),$$
and thus $\gamma(x,R2^{-k})\lesssim \gamma(p,r_k)$.  It follows that
\begin{align}
\notag  I
  &\lesssim \sum_{k=0}^\infty \sum_{p\in \cN_k} \cH^{2n+1}\bigl(B(p,R2^{-k})\cap \Gamma\bigr) \gamma(p,r_k) \\ 
\label{eq:graph discrete sgl}  &\lesssim \sum_{k=0}^\infty \sum_{p\in \cN_k} R^{2n+1}2^{-k(2n+1)} \gamma(p,r_k),
\end{align}
that is, $I$ is bounded by a sum over the quasiballs in the $\cC_k$'s.  

Next, we write the sum in \eqref{eq:graph discrete sgl} as an integral over slices. 
For notational convenience, we shorten $x_n^u$ to just $u$.  For nonempty sets $S\subset uP$, let
$$\theta^u(S)=\diam(S)^{-2n-2} \inf_{g \in \Aff} \|f-g\|^2_{L_2(S)},$$
and let $\theta^u(\emptyset)=0$.  By Section~\ref{sec:slices}, $uP$ is isomorphic to $\H_{n-1}$.  

Since $V_0$ is a product,
\begin{align*}
  \gamma(x,r) &=r^{-2n-1} \int_\R  \min_{g \in \Aff} \frac{\|f-g\|^2_{L_2(Q(x,r) \cap uP)}}{r^2} \ud u \\
  &= r^{-1} \int_\R \frac{\diam(Q(x,r) \cap uP)^{2n+2}}{r^{2n+2}} \theta^u(Q(x,r) \cap uP)\ud u\\
  &\approx r^{-1} \int_\R  \theta^u(Q(x,r) \cap uP)\ud u,
\end{align*}
where the last line follows from the fact that $Q(x,r) \cap uP$ is either empty or an $r$--quasiball (Lemma~\ref{lem:slices are quasi}).  Therefore,
\begin{align}\label{eq:sliced sgl}
  I
  &\lesssim \int_\R \sum_{k=0}^\infty \sum_{Q\in \cC_k} R^{2n}2^{-2nk} \theta^u(Q \cap uP)\ud u.
\end{align}

Finally, we apply Corollary~\ref{c:FO} in each slice and integrate.
We can slice each cover $\cC_k$ into a family of bounded-multiplicity covers of  slices.
Let $\cC^u_k=\{Q\cap uP:Q\in \cC_k\}\setminus\{\emptyset\}$.  This is a set of quasiballs in $uP$ with bounded multiplicity, and by Fubini's Theorem, we can rewrite \eqref{eq:sliced sgl} as
$$I \lesssim \int_\R \sum_{k=0}^\infty \sum_{S\in \cC^u_k} R^{2n}2^{-2nk} \theta^u(S)\ud u.$$

Let $a>0$ be such that $\diam S\le a r_k$ for any $S\in \cC^u_k$ and let $F_u=f|_{uP}$.  Then for any $S\in \cC^u_k$ and any $s\in S$, we have $\theta^u(S)\lesssim \theta_{F_u}(B_{uP}(s,ar_k)),$ where $\theta_{F_u}$ is as in in \eqref{eq:define theta}.  Since $\bigcup_{S\in \cC^u_k} S\subset Q(\zero,DR)$, the bounded multiplicity of $\cC^u_k$ implies
\begin{align*}
  \sum_{S\in \cC^u_k} R^{2n}2^{-2nk} \theta^u(S) 
  &\lesssim \sum_{S\in \cC^u_k}R^{2n}2^{-2nk} \cH^{2n}(S)^{-1} \int_S \theta_{F_u}(B_{uP}(s,ar_k))\ud s \\
  &\lesssim \int_{Q(\zero,DR) \cap uP} \theta_{F_u}(B_{uP}(s,ar_k))\ud s.
\end{align*}
This is zero when $|u|>DR$, so 
\begin{align*}
  I 
  & \lesssim \int_{-DR}^{DR} \sum_{k=0}^\infty \int_{Q(\zero, DR) \cap uP} \theta_{F_u}(B_{uP}(s,ar_k))\ud s \ud u\\
  & \lesssim \int_{-DR}^{DR} \int_{0}^{2ar_0} \int_{Q(\zero, DR) \cap uP} \theta_{F_u}(B_{uP}(s,r))\ud s \frac{\ud r}{r} \ud u.
\end{align*}
By Lemma \ref{lem:slices are quasi}, there is a $D'>0$ such that for each $u\in [-DR,DR]$, there is a ball $B_u\subset uP$ of radius $D'R$ such that $Q(p_0, DR) \cap uP\subset B_u$.  Thus, by Corollary \ref{c:FO},
\begin{align*}
  I &\lesssim \int_{-DR}^{DR} \int_0^{\infty} \int_{B_u} \theta^u(B_{uP}(s,r)) \ud s \frac{\ud r}{r}\ud u \\
  &\lesssim \int_{-DR}^{DR} \|F_u\|_{\lip}^2 (D'R)^{2n} \ud u \\
  &\lesssim_{\lambda} 2DR\cdot (D'R)^{2n} \approx R^{2n+1}.
\end{align*}
That is, 
\begin{align*}
  \int_0^R r^{-2n-1} \int_{B \cap \Gamma} \inf_{g \in \SAff_{P}} \frac{\|f - g\|^2_{L_2(Q(x,c^2r))}}{r^2} \ud x \frac{\ud r}{r} \lesssim_\lambda R^{2n+1}.
\end{align*}

Similarly, we can prove
\begin{align*}
  \int_0^R r^{-2n-1} \int_{B \cap \Gamma} \inf_{g \in \SAff_{P_i}} \frac{\|f_{w_i} - g\|^2_{L_2(Q_i(x,c^2r))}}{r^2} \ud x \frac{\ud r}{r} \lesssim_\lambda R^{2n+1}.
\end{align*}

This, together with \eqref{e:nonpar-par} allows us to conclude
\begin{align*}
  \int_{B \cap \Gamma} \int_0^R \beta_\Gamma(x,r)^2 \frac{\ud r}{r} \ud x \lesssim_\lambda R^{2n+1},
\end{align*}
as desired.
\end{proof}

\section{Slicing vertical functions}\label{sec:slice affine}

In this section, we prove bounds on functions that are close to affine on many families of parallel hyperplanes.  The simplest version of this bound deals with functions on the cube that are close to affine on any axis-parallel hyperplane.  For $V$ a vector space and $U\subset V$ a subspace, we define $\SAff_U(V)$ to be the set of measurable functions that are affine on every coset of $U$.
\begin{prop}\label{prop:slicing cube}
  Let $d\ge 3$ and let $R_1,\dots, R_{d}\subset \R^d$ be the coordinate $(d-1)$--planes, so that $R_i$ consists of points whose $i$th coordinate is zero.  Let $I=[-1,1]$ and let $\SAff_{R_i}(\R_d)\subset L_2(I^d)$ be the set of functions that are affine on each plane parallel to $R_i$.  Then, for any $g\in L_2(I^d)$,
  \begin{equation}\label{eq:slicing cube}
    \min_{\lambda\in \Aff} \|g-\lambda\|_2 \approx_d \sum_i \min_{\lambda\in \SAff_{R_i}(\R^d)} \|g-\lambda\|_2.
  \end{equation}
\end{prop}

This proposition implies the bounds on slice-affine functions on $\H_n$ used in the proof of Theorem~\ref{thm:beta2 highdim}.
\begin{lemma}\label{lem:slicing vertical}
  Let $n\ge 2$ and let $P_1, \dots, P_{2n-1}\subset V_0$ be vertical hyperplanes in general position, i.e., vertical $(2n-1)$--planes such that $\bigcap_i P_i=\langle Z\rangle$.  There is a $c>1$ such that for any vertical function $f\in L_2(V_0)$, any $x\in \Gamma_{f}$, and any $r>0$,
  \begin{equation}\label{eq:slicing vertical}
    \min_{g\in \Aff} \|f-g\|_{L_2(Q(x,r))} \lesssim_{P_1,\dots, P_{2n-1}} \sum_i \min_{g\in \SAff_{P_i}(V_0)} \|f-g\|_{L_2(Q(x,cr))},
  \end{equation}
  where for convenience we define $Q(x,r)=Q_{Y_n}(x,r)$.
\end{lemma}
\begin{proof}
After applying a translation, we may suppose that $x=\mathbf{0}$.

Let $A_0=A\cap V_0$ and let $P=P_{Y_n}=Y_n^{\overline{\Omega}}\cap V_0$.  Then $Q(x,r)=\Pi_{Y_n}(\delta_r(K))$
where 
$$K=\{s X_n + a + t Z: a\in B_A(x,1)\cap P, |s|\le 1, |t|\le 1\}.$$
Let $S_r=\pi(Q(x,r))$; this is isometric to the product of an $r$--ball and an interval of length $2r$, so it is a quasiball of radius $r$ in $A_0$.  Furthermore, for any $p\in S_r$, the intersection $\pi^{-1}(p)\cap Q(x,r)$ is an interval of length $2r^2$.  

We identify $A_0$ with $\R^{2n-1}$ by an isomorphism.  The projections $\pi(P_i)$ are subspaces of $A_0$ in general position, so there is a linear transformation $M\from A_0\to A_0$ such that $M(\pi(P_i))=R_i$ for all $i$.  Let $D_t=[-t,t]^{2n-1}\subset A_0$.  This is a quasiball in $A_0$, and $MS_r$ and $M^{-1}S_r$ are quasiballs with constants depending on $M$, so there is a $b>0$ depending only on $M$ such that $M S_t\subset D_{bt}$ and $D_t\subset M^{-1} S_{bt}$ for all $t>0$.

Since any $g\in \Aff$ is a vertical function,
$$\|f-g\|^2_{L_2(Q(x,r))}=2r^2 \|f-g\|^2_{L_2(S_r)}.$$
Let $h\from A_0\to \R$ be the function $h(v)=f(M^{-1}v)$; this is also a vertical function, and by a change of variables,
\begin{align*}
\min_{\lambda\in \Aff} \|f-\lambda\|^2_{L_2(Q(x,r))} 
& = \min_{\lambda\in \Aff} 2 r^2 \|f-\lambda\|^2_{L_2(S_r)} \\
& = \min_{\lambda\in \Aff} 2 r^2 |\det(M)|^{-1}\cdot \|h-\lambda\circ M^{-1}\|^2_{L_2(MS_r)}\\
& \le \min_{\lambda\in \Aff} 2 r^2 |\det(M)|^{-1} \|h-\lambda\|^2_{L_2(D_{br})}.
\end{align*}
By Proposition~\ref{prop:slicing cube}, 
$$\min_{\lambda\in \Aff} \|h-\lambda\|^2_{L_2(D_{br})} \lesssim \sum_i \min_{\lambda \in \SAff_{R_i}(\R^{2n-1})} \|h-\lambda\|^2_{L_2(D_{br})}.$$
Let $J_i$ be the determinant of the restriction $M^{-1}|_{R_i}\from R_i\to \pi(P_i)$.  If $\lambda \in \SAff_{R_i}(\R^{2n-1})$, then $\lambda \circ M \in \SAff_{\pi(P_i)}(A_0)$ and $\lambda \circ M \circ \pi \in \SAff_{P_i}(V_0)$, so by another change of variables,
\begin{align*}
\min_{\lambda \in \SAff_{R_i}(\R^{2n-1})} \|h-\lambda\|^2_{L_2(D_{br})}
& = \min_{\lambda \in \SAff_{R_i}(\R^{2n-1})} |J_i|^{-1} \|f-\lambda\circ M\|^2_{L_2(M^{-1} D_{br})}\\
& \le \min_{\lambda \in \SAff_{\pi(P_i)}(A_0)} |J_i|^{-1} \|f-\lambda\|^2_{L_2(S_{b^2r})}\\
& = 2 b^{-4} r^{-2} \min_{\lambda \in \SAff_{P_i}(V_0)} |J_i|^{-1} \|f-\lambda\|^2_{L_2(Q(x,b^2r))}.
\end{align*}
Combining these calculations, we find
\begin{align*}
\min_{\lambda\in \Aff} \|f-\lambda\|^2_{L_2(Q(x,r))} &\lesssim_{M} \min_{\lambda\in \Aff} r^2 \|h-\lambda\|^2_{L_2(D_{br})} \\ &\lesssim \sum_i r^2 \min_{\lambda \in \SAff_{R_i}(\R^{2n-1})} \|h-\lambda\|^2_{L_2(D_{br})}\\ &\lesssim_M b^{-4} r^{-2} \sum_i \min_{\lambda \in \SAff_{P_i}(V_0)} r^2 \|f-\lambda\|^2_{L_2(Q(x,b^2r)}.
\end{align*}
Since $b$ and $M$ depend on $P_1,\dots, P_n$, we conclude
$$\min_{\lambda\in \Aff} \|f-\lambda\|^2_{L_2(Q(x,r))} \lesssim_{P_1,\dots, P_n} \sum_i \min_{\lambda \in \SAff_{P_i}(V_0)} \|f-\lambda\|^2_{L_2(Q(x,b^2r)}.$$
as desired.
\end{proof}

In the rest of this section, we  prove Proposition~\ref{prop:slicing cube} using a wavelet decomposition.  Let $\psi_{j,k}\in L_2([-1,1])$ be the Haar wavelet basis for $L_2([-1,1])$.  That is, let $\psi_{0,0}(t)=1$, and for $j>0$ and $0\le k< 2^{i-1}$, let
$$\psi_{j,k}(t)=\begin{cases}
  -1 & 2k\cdot 2^{-j+1}-1 \le t < (2k+1) \cdot 2^{-j+1}-1 \\
  1 & (2k+1)\cdot 2^{-j+1}-1 \le t < (2k+2) \cdot 2^{-j+1}-1\\
  0 & \text{otherwise}.
\end{cases}$$
When $j>0$, the support of $\psi_{j,k}(t)$ is an interval of width $2^{-j+2}$.  

For multi-indices $\mathbf{j}=(j_1,\dots, j_d),\mathbf{k}=(k_1,\dots,k_d)\in \Z^d$ with $0\le k_i<\max \{1,2^{j_i-1}\}$, let
$$\Psi_{\mathbf{j},\mathbf{k}}(t_1,\dots, t_d) = \prod_{i=1}^d \psi_{j_i,k_i}(t_i).$$
Let $\supp(\mathbf{j})=\{i\mid j_i\ne 0\}$.  It is straightforward to check that the $\Psi_{\mathbf{j},\mathbf{k}}$'s are orthogonal.

Let
$$c_{\mathbf{j},\mathbf{k}}=\frac{\langle g, \Psi_{\mathbf{j},\mathbf{k}}\rangle}{\langle \Psi_{\mathbf{j},\mathbf{k}}, \Psi_{\mathbf{j},\mathbf{k}}\rangle}$$
so that 
$$g=\sum_{\mathbf{j},\mathbf{k}} c_{\mathbf{j},\mathbf{k}} \Psi_{\mathbf{j},\mathbf{k}}.$$

We partition this sum in two ways.  For every subset $S\subset \{1,\dots, d\}$, define
\begin{equation}\label{eq:subset decomp}
f_S=\mathop{\sum_{\mathbf{j},\mathbf{k}}}_{\supp(\mathbf{j})=S} c_{\mathbf{j},\mathbf{k}} \Psi_{\mathbf{j},\mathbf{k}}.
\end{equation}
For each $0\le i \le d$, let
$$g_i=\mathop{\sum_{\mathbf{j},\mathbf{k}}}_{|\supp(\mathbf{j})|=i} c_{\mathbf{j},\mathbf{k}} \Psi_{\mathbf{j},\mathbf{k}}=\sum_{|S|=i} f_S
.$$

Note that $f_S(t_1,\dots, t_d)$ depends only on the $t_i$ such that $i\in S$; in particular, $f_\emptyset=g_0$ is constant.  The $f_S$'s are pairwise orthogonal, as are the $g_i$'s.

Let $\lambda\from L_2([0,1]^d)\to \Aff$ be the orthogonal projection to $\Aff$.  Let $\SAff_i=\SAff_{R_i}(\R^d)$ and let $\lambda_i$ be the orthogonal projection to $\SAff_{i}$.  For every $\mathbf{j},\mathbf{k}$ with  $|\supp\mathbf{j}|\ge 2$, we have $\Psi_{\mathbf{j},\mathbf{k}}\perp \Aff$. Likewise, for every $\mathbf{j},\mathbf{k}$ with  $|\supp\mathbf{j}\setminus \{i\}|\ge 2$, we have $\Psi_{\mathbf{j},\mathbf{k}}\perp \SAff_i$.  Therefore,  
\begin{equation}\label{eq:lambda breakdown}
\lambda(g)=\lambda(g_0+g_1)=g_0+\lambda(g_1),
\end{equation}
and
\begin{equation}\label{eq:lambda i breakdown}
\lambda_i(g)=\lambda_i(g_0+g_1+g_2)=g_0+\lambda_i(g_1+g_2)
\end{equation}
for all $i$.

We will need two lemmas.
\begin{lemma}\label{lem:slice regression stuff}
Let $g_i$ and $f_S$ be as above.  For any $1\le l,m \le d$, $l\ne m$, we have $\lambda_l(f_{\{l\}})=f_{\{l\}}$, 
$\lambda_m(f_{\{l\}})=\lambda(f_{\{l\}})$, and 
$$\lambda_l(g_1)=f_{\{l\}} + \sum_{m\ne l} \lambda(f_{\{m\}}).$$
\end{lemma}
\begin{proof}
Since $f_{\{l\}}(t_1,\dots,t_l)$ depends only on $t_l$, it is constant on each plane parallel to $R_l$.  Therefore, $\lambda_l(f_{\{l\}})=f_{\{l\}}$.  

When $m\ne l$, then $f_{\{l\}}$ is independent of $t_m$, so $\lambda_m(f_{\{l\}})$ is independent of $t_m$ and affine on every plane parallel to $R_m$.  Therefore, $\lambda_m(f_{\{l\}})$ is affine and $\lambda_m(f_{\{l\}})=\lambda(f_{\{l\}})$.  
This implies
$$\lambda_l(g_1)=\sum_m\lambda_l( f_{\{m\}})=\lambda_l(f_{\{l\}}) + \sum_{m\ne l} \lambda(f_{\{m\}}).$$
\end{proof}

\begin{lemma}\label{lem:regression bounds}
\begin{equation}\label{eq:g regression}
  \|g-\lambda(g)\|_2^2
  =\| g_1-\lambda(g_1) \|^2 + \sum_{i=2}^\infty \|g_i\|^2,
\end{equation}
and for any $l$, 
\begin{align}\label{eq:g sliced regression}
  \|g-\lambda_l(g)\|^2
  &= \biggl\| g_1-\lambda_l(g_1) + g_2 - \lambda_l(g_2) + \sum_{i=3}^\infty g_i\biggr\|^2 \\ 
\notag  &= \left\| g_1-\lambda_l(g_1)\right\|^2 + \left\|g_2 - \lambda_l(g_2)\right\|^2 + \sum_{i=3}^\infty \left\| g_i\right\|^2.
\end{align}
\end{lemma}
\begin{proof}
By \eqref{eq:lambda breakdown} and the orthogonality of the $g_i$'s,
$$\|g-\lambda(g)\|_2^2=\left\|g_1-\lambda(g_1)+ \sum_{i=2}^\infty g_i\right\|_2^2= \|g_1-\lambda(g_1)\|_2^2+ \sum_{i=2}^\infty \|g_i\|_2^2 + 2 \sum_{i=2}^\infty\langle g_1-\lambda(g_1), g_i \rangle
.$$
For $i\ge 2$, $g_i$ is orthogonal to $g_1$ and to $\Aff$, so $\langle g_1-\lambda(g_1), g_i \rangle=0$.  This implies \eqref{eq:g regression}.

By \eqref{eq:lambda i breakdown},
$$g-\lambda_l(g)=(g_1-\lambda_l(g_1)) + (g_2 - \lambda_l(g_2)) + \sum_{i=3}^\infty g_i.$$
We claim that the terms in this sum are orthogonal.  When $i\ge 3$, $g_i$ is orthogonal to $g_1, g_2,$ and $\SAff_{l}$, so it suffices to check that $\langle g_1-\lambda_l(g_1), g_2 - \lambda_l(g_2)\rangle=0$.

Since  $(g_1 - \lambda_l(g_1))\perp \SAff_{l}$ and $\langle g_1,g_2\rangle=0$,
$$\langle g_1-\lambda_l(g_1), g_2 - \lambda_l(g_2)\rangle =\langle g_1-\lambda_l(g_1), g_2\rangle = -\langle \lambda_l(g_1), g_2\rangle.$$

By Lemma~\ref{lem:slice regression stuff}, $$\langle\lambda_l(g_1),g_2\rangle=\langle \alpha, g_2\rangle+\langle f_{\{l\}},g_2\rangle,$$
where $\alpha=\sum_{m\ne l} \lambda(f_{\{m\}})$ is affine.  Since $g_2$ is a sum of $\Psi_{\mathbf{j},\mathbf{k}}$'s with $|\supp(\mathbf{j})|=2$, it is orthogonal to any affine function.  Likewise, $f_{\{l\}}$ is a sum of $\Psi_{\mathbf{j},\mathbf{k}}$'s with $|\supp(\mathbf{j})|=1$, so $f_{\{l\}}$ is orthogonal to $g_2$.  Therefore,
$$\langle g_1-\lambda_l(g_1), g_2 - \lambda_l(g_2)\rangle=-\langle \lambda_l(g_1), g_2\rangle=0,$$
as desired.
\end{proof}

Next, we use these lemmas to show  that 
\begin{equation}\label{eq:g1 bound}
  \|g_1-\lambda(g_1)\|^2 \approx_d  \sum_{l=1}^d \|g_1-\lambda_l(g_1)\|^2
\end{equation}
and
\begin{equation}\label{eq:g2 bound}
   \|g_2\|^2 \approx_d \sum_{l=1}^d \|g_2-\lambda_l(g_2)\|^2.
\end{equation}


To prove \eqref{eq:g1 bound}, we decompose the left and right sides.  For the left side, note that $g_1=\sum_i f_{\{i\}}$, so 
$$g_1-\lambda(g_1)=\sum_i f_{\{i\}}-\lambda(f_{\{i\}}).$$
Each function $f_{\{i\}}(t_1,\dots,t_d)-\lambda(f_{\{i\}})(t_1,\dots,t_d)$ depends only on $t_i$ and satisfies
$$\int_{[-1,1]^d}f_{\{i\}}-\lambda(f_{\{i\}})\;d\mu = 0,$$
so when $i\ne j$,
$$\langle f_{\{i\}}-\lambda(f_{\{i\}}), f_{\{j\}}-\lambda(f_{\{j\}})\rangle=0.$$
Therefore, 
$$\|g_{1}-\lambda(g_{1})\|^2 = \sum_m\|f_{\{m\}} - \lambda(f_{\{m\}})\|^2.$$

For the right side, by Lemma~\ref{lem:slice regression stuff},
$$g_1-\lambda_l(g_1) = \sum_{m\ne l}\left(f_{\{m\}}-\lambda(f_{\{m\}})\right).$$
As noted above, the terms in the sum are pairwise orthogonal, so 
$$\|g_1-\lambda_l(g_1)\|^2 = \sum_{m\ne l} \left\|f_{\{m\}}- \lambda(f_{\{m\}})\right\|^2.$$
Summing over $l$, we get
$$\sum_l \|g_1-\lambda_l(g_1)\|^2 = \sum_{m} (d-1) \left\|f_{\{m\}}- \lambda(f_{\{m\}})\right\|^2 = (d-1) \|g_{1}-\lambda(g_{1})\|^2,$$
which proves \eqref{eq:g1 bound}.

To prove \eqref{eq:g2 bound}, note that if $|S\setminus \{l\}|\ge 2$, then $f_S$ is orthogonal to $\SAff_{l}$.  Let
$$h_l=\mathop{\sum_{S\subset \{1,\dots, d\}\setminus \{l\}}}_{|S|=2} f_{S}$$
and $k_l=g_2-h_l$, so that $h_l$ is orthogonal to $k_l$ and to $\SAff_{l}$.

Then
$$\|g_2-\lambda_l(g_2)\|^2 = \|h_{l}+k_{l}-\lambda_l(g_2)\|^2= \|h_{l}\|^2 +\|k_{l}-\lambda_l(g_2)\|^2\ge \|h_l\|^2$$
and
$$\sum_l \|g_2-\lambda_l(g_2)\|^2\ge \sum_l \|h_l\|^2 =\sum_l \mathop{\sum_{S\subset \{1,\dots, d\}\setminus \{l\}}}_{|S|=2} \|f_{S}\|^2.$$
Each subset $S$ of order $2$ appears $d-2$ times in the sum on the right, so
$$\sum_l \|g_2-\lambda_l(g_2)\|^2\ge (d-2) \sum_{|S|=2} \|f_{S}\|^2=(d-2) \|g_2\|^2.$$
Conversely, $\|g_2-\lambda_l(g_2)\|\le \|g_2-\lambda(g_2)\|$, so $\sum_l \|g_2-\lambda_l(g_2)\|^2\le d \|g_2-\lambda(g_2)\|^2$.  This proves \eqref{eq:g2 bound}.

Finally, we prove Proposition~\ref{prop:slicing cube}.  By \eqref{eq:g regression},
$$\min_{\lambda\in \Aff} \|g-\lambda\|^2 = \|g-\lambda(g)\|^2 = \| g_1-\lambda(g_1) \|^2 + \sum_{i=2}^\infty \|g_i\|^2.$$
By \eqref{eq:g sliced regression}--\eqref{eq:g2 bound},
\begin{align*}
  \sum_i \min_{\lambda\in \SAff_i} \|g-\lambda\|_2 
  &=\sum_l \|g-\lambda_l(g)\|^2\\ 
  &= \sum_l \left(\left\| g_1-\lambda_l(g_1)\right\|^2 + \left\|g_2 - \lambda_l(g_2)\right\|^2 + \sum_{i=3}^\infty \left\| g_i\right\|^2\right) \\
  & \approx_d \left\| g_1-\lambda(g_1)\right\|^2 + \left\|g_2\right\|^2 + \sum_{i=3}^\infty \left\| g_i\right\|^2 \\ 
  &= \min_{\lambda\in \Aff} \|g-\lambda\|^2,
\end{align*}
as desired.

\bibliographystyle{alpha}
\bibliography{slicing}
\end{document}